\documentclass[letter, 10pt]{article}

\usepackage[utf8]{inputenc}
\usepackage{hyperref}
\usepackage[nottoc,numbib]{tocbibind}
\usepackage{amsmath}
\usepackage{bm}
\usepackage{amsfonts}
\usepackage{amssymb}
\usepackage{amsthm}
\usepackage{makecmds}
\usepackage{todonotes}
\usepackage{graphicx}
\usepackage{subcaption}
\usepackage{mathtools}
\usepackage{centernot}
\usepackage{diagbox}

\usepackage{pgf}
\usepackage{tikz}
\usetikzlibrary{automata,positioning,arrows}

\numberwithin{equation}{section}

\newtheorem{theorem}{Theorem}[section]
\newtheorem{lemma}[theorem]{Lemma}
\newtheorem{prop}[theorem]{Proposition}
\newtheorem{conj}[theorem]{Conjecture}

\newtheorem{defn}{Definition}[section]
\newtheorem{probenv}{Problem}[section]

\newtheorem{corollary}{Corollary}[theorem]

\newenvironment{varscope}{}{}

\usepackage{xspace}	
\newcommand*{\eg}{e.g.\@\xspace}
\newcommand*{\ie}{i.e.\@\xspace}

\makeatletter
\newcommand*{\etc}{%
    \@ifnextchar{.}%
        {etc}%
        {etc.\@\xspace}%
}
\newcommand*{\etal}{%
    \@ifnextchar{.}%
        {et al}%
        {et al.\@\xspace}%
}
\newcommand*{\iid}{%
    \@ifnextchar{.}%
        {i.i.d}%
        {i.i.d.\@\xspace}%
}
\makeatother

\makecommand{\defneq}{\triangleq}

\makecommand{\nats}{{\mathbb N}}
\makecommand{\ints}{{\mathbb Z}}

\makecommand{\perm}{\sigma}
\makecommand{\varperm}{\varsigma}
\makecommand{\pile}{y}

\makecommand{\phase}{t}
\makecommand{\numphase}{T}

\makecommand{\numpile}{m}
\makecommand{\ordervar}{\psi}

\makecommand{\embed}{\phi}

\makecommand{\qshuffle}{{\rm qshuffle}}
\makecommand{\qshuffle}{{\rm \qsym {\text -} shuffle}}
\makecommand{\sshuffle}{{\rm \ssym {\text -} shuffle}}

\makecommand{\pilesfn}{{\rm piles}}

\makecommand{\qphase}{{\rm multi {\text -} phase}}
\makecommand{\sphase}{{\rm s {\text -} phases}}
\makecommand{\digit}{{\rm d}}
\makecommand{\Numpiles}{M}
\makecommand{\twodots}{\mathinner {\ldotp \ldotp}}
\makecommand{\schedules}{{\mathbf y}}
\makecommand{\cumread}{{\rm cumread}}

\makecommand{\shuffleOp}{{\rm shuffle}}
\makecommand{\pileType}{\chi}
\makecommand{\pileTypeInd}{\kappa}
\makecommand{\assignfn}{\pile}
\makecommand{\assignFn}{\assignfn}
\makecommand{\pileIndex}{p}
\makecommand{\condReverseFn}{r}
\makecommand{\assignTypeFn}{z}
\makecommand{\assignTypeSum}{Z}

\makecommand{\numPileSeq}{M}
\makecommand{\numPile}{\numpile}
\makecommand{\numPhase}{\numphase}
\makecommand{\numAccum}{\nats}
\makecommand{\reals}{{\mathbb R}}

\makecommand{\reversePerm}{\perm_{\rm J}}

\makecommand{\pileTypeVec}{{\boldsymbol\chi}}
\makecommand{\pileAssignments}{{\mathcal X}}

\makecommand{\indord}{{\rm ord}}

\makecommand{\arityFn}{{\rm arity}}
\makecommand{\xor}{\oplus}

\makecommand{\permsize}{n}
\makecommand{\nclause}{{n_{\rm c}}}
\makecommand{\nvar}{{n_{\rm v}}}
\makecommand{\radix}{r}
\makecommand{\logict}{{\rm T}}
\makecommand{\logicf}{{\rm F}}

\makecommand{\qsym}{{\mathbf q}}
\makecommand{\ssym}{{\mathbf s}}
\makecommand{\arbpile}{{\mathcal P}}

\def\NPHard/{NP-Hard}
\def\POLY/{P}

\title{Sorting permutations with pile shuffle on queue-like and stack-like piles}
\author{Kyle B. Treleaven \\
\small \texttt{ktreleav@alum.mit.edu}
}
\date{}  
\date{\today} 
\date{May 30, 2025}

\begin{document}

\maketitle

\begin{abstract}

Inspired by a common technique for shuffling a deck of cards on a table without riffling,
we formalize the pile shuffle and investigate its capabilities as a sorting device.
Our study is novel in that
we consider pile shuffle in three variations:
(1) using queue-like piles,
(2) using stack-like piles, and
(3) using a heterogeneous mixture of those two pile types,
either given or decided during the shuffle by the dealer herself.
We
characterize the sortable permutations
of one or more sequential rounds of pile shuffle,
under constraints on the number and types of piles to be used.
Meanwhile, 
we derive formulas to obtain a single sorting shuffle
of a given permutation
on the minimum number of piles, efficiently, whenever feasible.
Next, we present an interesting mathematical framework for interpreting
multi-round sequential shuffles
as singular shuffles on a generally larger number of ``virtual piles'',
lifting most of our previous results to the multi-round setting;
we confirm that repetition augments the power of pile shuffle exponentially, in that
$m$ piles over $T$ rounds enjoys the capacity of $m^T$ piles in a single round.
Finally, we motivate a forthcoming companion paper
in which we prove that
dealer choice---%
where the dealer is allowed to choose the types of the piles during the shuffle---%
makes deciding feasibility of sort \NPHard/
for some variants of multi-round pile shuffle;
proof is by a novel reduction from Boolean satisfiability (SAT).

\end{abstract}



\section{Introduction}
\label{sec:intro}

Pile shuffle is a common method for shuffling a deck of cards on a table without riffling,
and is more or less a non-mechanical version
of the \emph{shelf shuffle}~\cite{FULMAN2021112448} used in shuffling machines in casinos:
A deck of cards is dealt into piles on an empty card table, and
the resulting piles are stacked up in some order to form the new deck.
The process may be repeated until the deck is ``shuffled enough''.

While the majority
of the literature on card shuffling techniques
is dedicated to randomization,
in this paper
we investigate the capabilities of
pile shuffle as a sorting device.
The study pertains to randomization, however, by recognizing that
good mixing might be achieved most directly
by combining a sorting device with a randomized labeler.
We consider pile shuffle in three variations:
(1) using queue-like piles, as when cards are flipped over as they are dealt,
(2) using stack-like piles, as when they are not flipped over during the deal, and
(3) using a heterogeneous mixture of those pile types,
preferably decided during the deal by the dealer herself.
As far as the author is aware, ours is the first study to consider
arbitrary mixtures of the two pile types.

Pile shuffle is limited as a sorting device in that
not every deck of cards can be sorted in a single ``round'' of shuffle on a bounded number of piles.
As a trivial example, the only deck that can be sorted using a single queue is the already-sorted deck.
Motivated by this observation,
our focus is on
(1) characterizing the sortable permutations
of one or more sequential rounds of pile shuffle,
under constraints on the number of rounds, and the number and types of piles to be used, and
(2) obtaining efficient algorithms to compute a sorting shuffle of a given permutation whenever feasible.

\emph{Contributions:}
The contributions of this paper are as follows:

We formulate a mathematical model of pile shuffle, and
present necessary and sufficient conditions for a single shuffle to sort an input permutation.
Our formulation is more or less the same thing as $P$-partitions~\cite{FULMAN2021112448},
but stated in a way most conducive to our  analysis.
From these conditions we obtain formulas for
the transcription of a sorting shuffle on the minimum number of piles (a \emph{minimal} sort)
whenever feasible.
Our formulas are efficient in that they can be computed
in a single scan of the input permutation,
in time that is therefore linear in its length.
In the homogeneous all-queues or all-stacks cases,
we reproduce formulas
for the minimum number of piles required,
in terms of well-studied ascent and descent permutation statistics~\cite{bona_combinatorics_2004, butler_stirling-euler-mahonian_2023}
with deep connections to combinatorics.
This allows us to reproduce basic results about the probability
that a random permutation is sortable under given constraints,
through a connection between permutation statistics and Eulerian numbers~\cite{petersen2015eulerian}.

In the multi-round setting,
we present a mathematical framework
which allows us to reinterpret any sequence of shuffles
as an equivalent ``virtual shuffle'' in a single round.
This mathematical equivalence lifts nearly all of our results
about single-round shuffle
to the multi-round setting.
We confirm that repetition augments the power of pile shuffle exponentially, in that
$m$ piles over $T$ rounds has the capacity of $m^T$ piles in a single round.
%
%

%
Finally,
we expose
non-trivial complexity in deciding sort feasibility
in the multi-round case, if
the dealer is allowed to choose the types of the piles arbitrarily during the shuffle.
(The single-round case remains tractable.)
We present two such sort feasibility problems of general interest
and motivate a forthcoming companion paper in which we demonstrate---%
by a novel reduction from Boolean satisfiability (SAT)---that
at least one of those variants is NP-Hard.

\subsection{Literature Review}

Sorting is one of the fundamental problems in computer science~\cite{Knuth:art3},
with deep connections to combinatorics and computational complexity theory.
It has attracted significant research since the beginning of computing,
no doubt because the theoretical properties of sorting algorithms have direct and significant impact
on the performance of large-scale computer systems.
The most popular branches of the literature
tend to focus on the typical random-access model of computer memory,
where data at a particular memory address may be mutated at unit cost.
In such setting the most common objectives are to obtain space- and time- efficient sort algorithms,
optimizing resource consumption.

Sorting problems based on the physical constraints of storage present other unique challenges, and
are also of both practical and theoretical significance.
In physically-motivated settings like pile shuffle,
even asking whether a collection of items can be sorted at all (feasibility)
can become an interesting question.

\emph{Patience sort.}
The present study
is arguably closest to that of Patience sort~\cite{chandramouli2014patience, burstein2006combinatorics},
in that it is similarly motivated by sorting in piles, rather than randomizing.
Patience sort
represents a greedy strategy for the so-called Floyd's Game, and
happens to be an efficient means to compute longest increasing subsequences.
Both Patience sort and sort by pile shuffle are so-called \emph{distribution sorts}~\cite[p.~168]{Knuth:art3}, in that
the sort occurs in a distribution phase, followed by a collection phase.
However, Patience sort is a kind of \emph{merge sort}---%
the piles are merged together into the output one card at a time---%
whereas pile shuffle sort is a so-called \emph{bucket sort}---%
the piles are concatenated, each as a whole, in a chosen order.
%
%

While bucket sorts classically employ some kind of sub-sort within each bucket (sometimes recursively)
that is not always physically practical.
%
Instead, pile shuffle sort is a non- sub-sorted, a.k.a. \emph{pure}, bucket sort.
In lieu of sub-sorting,
we rely on
multiple rounds of shuffle
to augment the power of a bounded number of piles,
at the expense of more time to execute the shuffle.
In the motivating case of sort with a physical facility (\eg, a table of a specific size)
that is often a necessary trade-off.

\emph{Sorting networks.}
There is a significant body of literature about the sortable permutations of so-called sorting \emph{networks}~%
\cite{bona_combinatorics_2004, adda7f6a9afa4f559318c219c37a8dfe, halperin_complexity_2008}.
%
While the study was founded originally on networks of queues and stacks, 
a wide variety of more complex shuffle devices---%
including ones inspired by card shuffling---%
have been studied~\cite{pudwell_sorting_2023,dimitrov_sorting_2022}.
Although bucket sorts like pile shuffle might also be cast in the network framework,
the fit is not entirely natural, and
the author is not aware of previous such treatments.

\emph{Pancake sort.}
Finally, the so-called ``Pancake'' sort~\cite{GATES197947} is another physically inspired problem,
of sorting a disordered stack of pancakes by repeatedly inserting a spatula at some point in the stack and flipping all pancakes above it.
Pancake sort appears in applications in parallel processor networks, and can provide an effective routing algorithm between processors~%
\cite{GARGANO1993315,4032188}.
Pancake sort has also been called an ``educational device'', and 
it was shown to be NP-Hard in~\cite{BULTEAU20151556} by a reduction from $3$SAT.
In the forthcoming companion paper in the series,
we prove NP-Hardness of multi-round heterogeneous variants of pile shuffle sort by reduction from SAT.

\subsection{Organization}

The rest of the paper is organized as follows.
We define notation and summarize necessary background for the paper in Section~\ref{sec:background}.
We formulate the pile shuffle mathematically
and state the objectives of the paper
in Section~\ref{sec:problem}.
We analyze our model and present
results pertaining to
a single round of pile shuffle---%
we cover shuffle on queues, on stacks, and on a mixture of pile types---%
in Sections~\ref{sec:qsort}, \ref{sec:stacksort}, and~\ref{sec:hetero-sort}, respectively.
(We briefly discuss sorting random permutations
in Section~\ref{sec:sort-random}.)
In Section~\ref{sec:multi-round},
we extend our single-round results
to the case of
multiple sequential rounds of pile shuffle.
Then
in Section~\ref{sec:dealer-choice-multi},
we expose the additional complexity
of allowing the dealer to choose the types of piles
in the multi-round setting,
motivating a forthcoming paper on the topic.
Finally, we summarize our results, and offer conclusions, in Section~\ref{sec:conclusion}.

\section{Background}
\label{sec:background}

In this section we introduce notation used throughout the paper, 
as well as
necessary mathematical background about ordered sets and permutations.

\subsection{Notation}

We use the following notation throughout the paper.

\emph{Ranges.}
We will denote by $\nats$ the natural numbers, and by $[n]$ the natural range $\{ 1, 2, \ldots, n \}$.
We will denote
by $[n] \pm k$ the shifted range $\{ i \pm k : i \in [n] \}$;
in particular, $[n] - 1 = \{ 0, 1, \ldots, n-1 \}$ is often useful.

\emph{Sequences.}
We will denote the set of all sequences of length $n$ on a set $S$ by $S^{[n]}$, or more lazily by $S^n$.
%
In some instances we will denote sequences in string form:
We may write $f = f_1 f_2 f_3$ to specify a sequence of three elements with $f(1) = f_1$, $f(2) = f_1$, $f(3) = f_3$.
We use superscripting in string form to denote repetition.
For example,
if $ABC$ is the partition of a sequence into three segments, then
$AB^3C = ABBBC$.

\emph{Function composition.}
We use the notation $h = f(g)$ for the composition of functions $f: Z \to Y$ and $g: X \to Z$,
\ie, $h(x) = f(g(x))$ for all $x \in X$.

\emph{Indicator expressions.}
In some equations we use a so-called indicator notation, where
$[P]$ denotes the indicator function associated with a proposition $P$.
That is, $[P] = 1$ if the proposition is true, otherwise $[P] = 0$.
For example $f(x) = [ x \geq 3 ]$ means that $f(x) = 1$ over $x \geq 3$ and $0$ elsewhere.

\emph{Modular arithmetic.}
We use 2-modulo arithmetic sparingly, with notation $a \xor b \doteq (a + b) \mod 2$.

\subsection{Ordered sets}

An \emph{ordered set} is a pair $(X, \preceq)$ of a set $X$ and a binary relation (the order) $\preceq$
that satisfies for all $a, b, c \in X$:
$a \preceq a$ (reflexivity),
$a \preceq b$ and $b \preceq a \implies a = b$ (antisymmetry), and
$a \preceq b$ and $b \preceq c \implies a \preceq c$ (transitivity).
If every two points are comparable, then $\preceq$ is a \emph{total order}.
Any order $\preceq$ can be equivalently expressed in terms of a strict order $\prec$ defined by $a \prec b \iff a \preceq b$ and $a \neq b$.
We read $a \prec b$ as ``$a$ precedes $b$''.

\subsection{Permutations}

\makecommand{\Perm}[1]{{[#1]!}}

A permutation of a finite set $X$ is a bijective mapping $\perm: X \to X$.
It is well known that there are $n!$ (factorial) permutations of a set of size $n$.
We will denote by $X!$ the set of all such permutations.
The identity permutation
is the special permutation $\perm_I$ where $\perm_I(x) = x$ for all $x \in X$.
A composition of two permutations is also a permutation.
The inverse of a permutation $\perm$, denoted $\perm^{-1}$, is the permutation with the property
$\perm^{-1}(\perm) = \perm(\perm^{-1}) = \perm_I$.

\section{Problem Statement}
\label{sec:problem}

In this section
we formalize the pile shuffle with a mathematical framework
used throughout our study,
and 
we state the objectives of the paper.

Pile shuffle begins
with a deck of cards and an empty card table, and 
has two phases---distribution (or ``the deal''), and collection:
During the deal, until the deck is empty,
we place the next card from the top of the deck onto the table,
either directly on the table, creating a new \emph{pile}, or
on top of another pile previously created.
During the collection phase, 
once the deck has been fully dealt out,
we pick up each pile as a whole,
one at a time in some order,
and add it to the bottom of the new deck.

\emph{Deck representation:}
We assume a fixed assignment of labels from $[n]$ to $n$ distinct elements of a deck, so that
any deck ordering can be represented by a permutation $\perm \in [n]!$.
In particular, we represent the order of a given deck
by the permutation
where
$\perm(s)$ is the position of label $s$ 
for every $s \in [n]$.
We call this the \emph{embedding} convention,
which is a departure from a perhaps more typical---and inverse---\emph{sequence} convention, where
$\perm(k)$ would be the label in position $k$.
We use the embedding convention throughout the study because of its property that
$s$ precedes $t$ in the deck if and only if $\perm(s) < \perm(t)$.
A deck is \emph{sorted} if it is represented by the identity permutation $\perm_I$.

\emph{Pile types:}
We consider piles of two types in this paper: queues and stacks.
A queue maintains the order of cards placed on it.
For example,
if we deal the sequence 1234 into a queue, and pick it back up, we obtain 1234 again.
On a card table, flipping the cards over during the deal creates this kind of behavior.
If label $s$ precedes $t$ in the input deck, and they are placed into the same queue together, then
$s$ precedes $t$ in the new deck also.
In contrast, a stack reverses the order of placement:
If $s$ precedes $t$ in the deck and they are placed into the same stack, then
$t$ precedes $s$ in the new deck;
if we deal 1234 into a stack, we obtain 4321 back.
Dealing cards into a pile \emph{without} flipping them over creates stack-like behavior.

\emph{Pile assignment:}
The outcome of a pile shuffle depends on
(1) the input permutation, given,
(2) the pile types used,
(3) the assignment of cards onto piles during the deal, and
(4) the order 
in which the piles are retrieved
during collection.
If we number the piles in the order that they are picked up,
then the assignment of labels to piles can be described
by a function $\pile: [n] \to \nats$ of \emph{pile assignments},
assigning each label $s \in [n]$
to the $\pile(s)$-th pile collected.
Pile shuffle ensures that
the new deck $\varperm$ obeys
$\pile(s) < \pile(t) \implies \varperm(s) < \varperm(t)$,
for every pair of labels $s$ and $t$;
this is true regardless of the pile type(s) that are used.

\emph{Example:}
Suppose we shuffle a deck with label sequence $\perm^{-1} = 456123$ using pile assignments $\pile = 421242$ on stacks.
We can imagine dealing from a face-up deck of cards,
and using $\pile$ to determine
which pile to place
each card
into.
Note that $\pile$
is a function of
a card's label only, and ignores its position in the deck.
First we place
item $\perm^{-1}(1) = 4$ into pile $\pile(4) = 2$, then
item $\perm^{-1}(2) = 5$ into pile $\pile(5) = 4$, and so on.
After the deal we will see piles as below.
\[
\left.
\begin{array}{cccc}
  & 2 & & \\
  & 6 & & 1 \\
3 & 4 & & 5 \\
\hline
P1 & P2 & P3 & P4
\end{array}
\right.
\]
Note the number of non-empty piles used during shuffle is equal to the number of \emph{distinct} pile assignments,
in this case three.
Collecting the piles in increasing order (left-to-right) we obtain the new deck $\varperm$;
in this case $\varperm^{-1} = 326415$.

\emph{Objectives:}
The present paper has two objectives.
First, we seek to characterize the sortable permutations of one or more sequential rounds of pile shuffle,
under constraints on the number of rounds and the number and types of piles to be used.
A permutation is sortable if there exists a permissible pile shuffle---%
\ie, the pile assignments in one or more sequential rounds---%
that causes the deck to become sorted.
Second, we seek efficient algorithms to compute a sort of an input permutation when feasible, or else 
identify that it is not sortable.

\section{Sorting with a pile shuffle on queues}
\label{sec:qsort}

We first study pile shuffle on queues, or \emph{queue shuffle}, 
which lays the foundation for the other cases.
While some of the results of these early sections may be known generally,
the techniques developed here are needed for later sections.
We begin by analyzing
our pile shuffle model
to derive conditions for a single shuffle on queues to sort an input permutation.
From those conditions we derive a simple lower bound on the number of queues required
to sort a given permutation, in terms of well-known permutation statistics.
Finally, we present a formula which constructs a minimal sort
(\ie, a sort on the minimum number of piles)
in time that is linear in the length of the permutation.

\subsection{Analysis}
Suppose that $\varperm \in \Perm{n}$ is the embedding representation of a deck
resulting by
shuffling a deck $\perm \in \Perm{n}$ on queues using pile assignments $\pile$.
We will denote the relation
\[
\qshuffle \, \pile \, \perm = \varperm.
\]
We recall that
for distinct labels $s, t \in [n]$, 
$s$ precedes $t$ in the output, \ie,
$\varperm(s) < \varperm(t)$, either
if its pile assignment is lower,
$\pile(s) < \pile(t)$,
or within the same pile
if $s$ precedes $t$ in the original deck,
$\perm(s) < \perm(t)$.
We can express this mathematically by
\makecommand{\outer}{\Big}
\begin{align}
\label{eq:qshuffle-relation-long}
\qshuffle \, \pile \, \perm = \varperm
\quad \iff \quad
\bigg[
\varperm(s) < \varperm(t)
\iff
\outer( \pile(s), \, \perm(s) \outer)
<
\outer( \pile(t), \, \perm(t) \outer)
\bigg]
.
\end{align}
(The final order is the lexicographical order.)


Relations like~\eqref{eq:qshuffle-relation-long} would become increasingly tedious to carry around
in the sequel, so
we will introduce a more concise notation based on relations between \emph{induced orders}:

If $f: X \to Y$ is an injective function
to an ordered set $(Y,\prec)$, then
we denote by $\prec_f$ the order induced on $X$ by
\begin{align}
\label{eq:induced-order}
x \prec_f x' \iff f(x) \prec f(x').
\end{align}
When $f$ is defined by an expression $f(x)$ we may write
$\prec_f$ as $\indord_x \, f(x)$.
Then if
$(Y_1, \prec_{1})$ and $(Y_2, \prec_{2})$ are two ordered sets, and
we have injections
$f_1 : X \to Y_1$ and $f_2: X\to Y_2$,
then
\begin{align}
\label{eq:induced-order-equality}
\indord_x \, f_1(x) = \indord_x \, f_2(x)
\quad
\iff 
\quad
\Big[
f_1(x) \prec_{1} f_1(x') \iff f_2(x) \prec_{2} f_2(x')
\Big]
.
\end{align}
We read the
the left-hand side as:
$f_1$ and $f_2$ induce the same order on $X$.

We use this condition 
to express~\eqref{eq:qshuffle-relation-long} as
\begin{align}
\label{eq:qshuffle-order-identity}
\qshuffle \, \pile \, \perm = \varperm
\quad \iff \quad
\indord_s \, \varperm(s)
\, = \,
\indord_s \, \outer( \pile(s), \, \perm(s) \outer)
.
\end{align}
In words,
a queue shuffle using pile assignments $\pile$ transforms $\perm$ into $\varperm$
if and only if $\varperm$ induces the same order on $[n]$ as $(\pile, \perm)$ does lexicographically.

\subsection{Sorting shuffles and feasibility}

\makecommand{\descfn}{{\rm desc}}
\makecommand{\readings}{{\rm read}}
\makecommand{\ascruns}{{\rm ascrun}}
\makecommand{\ascsFn}{{\rm ascs}}
\makecommand{\descrunFn}{{\rm descrun}}

A pile shuffle is a sort if it produces as output the identity permutation,
$\varperm = \perm_I$.
When sort
occurs,
$\perm_I(s) = s$ everywhere
simplifies~\eqref{eq:qshuffle-order-identity} to
\begin{align}
\label{eq:sort-order-identity-var}
\qshuffle \, \pile \, \perm = \perm_I
\quad \iff \quad &
\indord_s \, s \,
\, = \,
\indord_s \, \Big( \pile(s), \perm(s) \Big)
\\
\label{eq:qsort-order-var-logical}
\quad \iff \quad &
s < t
\iff
\Big( \pile(s), \perm(s) \Big)
<
\Big( \pile(t), \perm(t) \Big)
.
\end{align}
Our first result expresses this sort condition in an 
efficiently checkable
form.
\begin{lemma}[Sort with queues]
\label{lemma:sort-with-queues}
A pile assignment function $\pile$ is a sort of permutation $\perm \in \Perm{n}$ on queues
($\qshuffle \, \pile \, \perm = \perm_I$)
if and only if 
\begin{equation}
\label{eq:qsort-cond}
\pile(s+1)
\geq
\pile(s) + \left[ \perm(s+1) < \perm(s) \right]
\qquad \forall s \in [n-1]
.
\end{equation}
\end{lemma}
\begin{proof}
Due to transitivity, 
the right-hand side of~\eqref{eq:qsort-order-var-logical} is true if and only if
\begin{equation}
\label{eq:qsort-cond-adj}
\Big( \pile(s), \perm(s) \Big)
<
\Big( \pile(s+1), \perm(s+1) \Big)
\qquad \forall s \in [n-1]
.
\end{equation}
Therefore, we may prove the lemma by equating~\eqref{eq:qsort-cond} with~\eqref{eq:qsort-cond-adj}.
This can be done by cases:
Wherever $\perm(s) < \perm(s+1)$,
\eqref{eq:qsort-cond-adj} is satisfied if and only if $\pile(s+1) \geq \pile(s)$;
wherever $\perm(s+1) < \perm(s)$,
if and only if $\pile(s+1)$ is \emph{strictly} greater, \ie, $\pile(s+1) \geq \pile(s) + 1$.
\eqref{eq:qsort-cond} is precisely the combined expression of these two cases
using indicator notation.
\end{proof}
Unlike~\eqref{eq:qsort-order-var-logical},
the condition~\eqref{eq:qsort-cond} can be checked with a linear scan.
An intuition behind the condition is as follows:
A sort may never assign element $(s+1)$ to a lower-valued pile than $s$, \ie,
one this is collected earlier in the collection phase.
However, as long as
$s$ precedes $(s+1)$ during the deal, then
$(s+1)$ may be placed in any pile $\pile(s+1) \geq \pile(s)$.
Otherwise---%
mathematically, when
$\perm(s+1) < \perm(s)$---%
then $(s+1)$ must be placed into a \emph{strictly} higher-valued pile,
collected later, \ie,
$\pile(s+1) \geq \pile(s) + 1$.
For example, 
when dealing the sequence $1423$,
$2$ could be dealt into the same queue as $1$, since it is dealt after $1$,
but $3$ cannot be dealt into the same queue as $4$, since $4$ precedes it;
ultimately, 4 must be found in a higher-valued pile than 3, collected later.

\eqref{eq:qsort-cond} is equivalently stated:
\begin{enumerate}
\item $\pile$ is monotonically non-decreasing; and
\item it is strictly increasing at every descent of $\perm$.
\end{enumerate}
A descent of a sequence $x$ is a position $t$ where $x(t+1) < x(t)$.
The number of descents of a permutation,
\begin{equation}
\label{eq:num-descents}
\descfn(\perm) \defneq \sum_{s=1}^{n-1} \left[ \perm(s+1) < \perm(s) \right]
,
\end{equation}
is a well-studied permutation statistic~\cite[p.~4]{bona_combinatorics_2004}.
It is one less than the number of ascending runs, $\ascruns(\perm)$, where
the ascending runs of a sequence are the contiguous increasing subsequences that cannot be extended on either end.
The number of ascending runs of $\perm$ also corresponds to the number of 
``readings'' $\readings(\perm^{-1})$ of the deck sequence $\perm^{-1}$, which is
the number of times one must scan through $\perm^{-1}$ to find the numbers $1, 2, \ldots, n$
in order without ever backtracking.
For example the sequence $364152$ has three readings: $12$, $345$, and $6$.

\begin{lemma}
\label{lemma:sort-lower-bound}
If $\pile$ is the pile assignments of a queue shuffle sort of permutation $\perm \in \Perm{n}$,
and, without loss of generality, $\pile(1) = 1$, then
the number of piles $\pile(n)$ is bounded by
\begin{equation}
\label{eq:min-piles}
\pile(n)
\geq
1 + \descfn\left(\perm\right)
=
\ascruns\left(\perm\right)
=
\readings(\perm^{-1})
.
\end{equation}
\end{lemma}
\begin{proof}
We obtain~\eqref{eq:min-piles}
by
writing $\pile(n)$ as a telescoping sum
starting from $\pile(1) = 1$,
bounding the sum by rearranging~\eqref{eq:qsort-cond}, and then
substituting~\eqref{eq:num-descents} in the bound:
\begin{align*}
\pile(n) &= \pile(1) + \sum_{s=1}^{n-1} \pile(s+1) - \pile(s)
\\
&\geq
1 + \sum_{s=1}^{n-1} \left[ \perm(s+1) < \perm(s) \right]
= 1 + \descfn\left(\perm\right)
.
\end{align*}
\end{proof}

Finally we obtain a minimal sort in the sense of fewest piles used.
\begin{lemma}
\label{lemma:qsort-bound-realz}
Given a permutation $\perm \in [n]!$,
let $\pile^*$ be defined by
\begin{equation}
\label{eq:qsort-bound-realz}
\pile^*(s) = 1 + \sum_{s'=1}^{s-1} \left[ \perm(s'+1) < \perm(s') \right]
\qquad \forall s \in [n]
.
\end{equation}
We call it the cumulative ascending runs function of $\perm$.
$\pile^*$ is a minimal assignment function for sorting $\perm$ on queues.

\end{lemma}
\begin{proof}
It is easy to verify that
the definition~\eqref{eq:qsort-bound-realz} of $\pile^*$ satisfies~\eqref{eq:qsort-cond}
and that
$\pile^*(n) = 1 + \descfn(\perm)$
by construction,
matching the lower bound~\eqref{eq:min-piles}.
\end{proof}

\begin{corollary}
\label{cor:min-piles-qsort}
A permutation $\perm$ can be sorted with $\numpile$ queues if and only if 
$\ascruns(\perm) \leq \numpile$.
\end{corollary}
\begin{proof}
The corollary is an immediate result of the minimal construction~\eqref{eq:qsort-bound-realz}
achieving the lower bound~\eqref{eq:min-piles} in all instances.
\end{proof}

\subsection{Demonstration}

Lemma~\ref{lemma:sort-with-queues} provides
a necessary and sufficient condition
for a pile shuffle on queues to sort an input permutation.
The condition is efficiently checkable in a linear scan of the input permutation.
We have also derived the formula~\eqref{eq:min-piles} for the minimum number of queues needed to sort an input permutation,
in terms of the well-known permutation statistic of ascending runs.
A companion formula~\eqref{eq:qsort-bound-realz} produces a sort if feasible on the minimum possible number of queues,
in linear time.

We end this section with a demonstration of sorting with pile shuffle guided by Lemma~\ref{lemma:qsort-bound-realz}.
We start by writing an example permutation in the so-called two-line notation,
a two-row matrix where the permutation pre-image is enumerated across the top row, and
the matching image is written underneath it:
\[
\left\{\begin{array}{cccccccc}
1 & 2 & 3 & 4 & 5 & 6 & 7 & 8 \\
3 & 7 & 2 & 4 & 6 & 8 & 1 & 5
\end{array}\right\}
=
\left\{\begin{array}{ccc}
\ldots & s & \ldots
\\
\ldots & \perm(s) & \ldots
\end{array}\right\}
.
\]
We use the set-notation brackets rather than the typical parentheses to emphasize that the column order does not matter.
However, with the first row in the normal ascending order,
the action of~\eqref{eq:qsort-bound-realz} is easily shown by adding $\pile^*$ as a third row to the matrix
\[
\left\{\begin{array}{cccccccc}
1 & 2 & 3 & 4 & 5 & 6 & 7 & 8 \\
3 & 7 & 2 & 4 & 6 & 8 & 1 & 5 \\
1 & 1 & 2 & 2 & 2 & 2 & 3 & 3 
\end{array}\right\}
=
\left\{\begin{array}{ccc}
\ldots & s & \ldots
\\
\ldots & \perm(s) & \ldots
\\
\ldots & \pile^*(s) & \ldots
\end{array}\right\}
.
\]
It is easy to check that $\pile^*$ increments at each of the descents of $\perm$.

If we rearrange the columns so that the second row appears in ascending order, then
the sequence representation of the deck appears in the top row,
\[
\left\{\begin{array}{cccccccc}
7 & 3 & 1 & 4 & 8 & 5 & 2 & 6 \\
1 & 2 & 3 & 4 & 5 & 6 & 7 & 8 \\
3 & 2 & 1 & 2 & 3 & 2 & 1 & 2 
\end{array}\right\}
=
\left\{\begin{array}{ccc}
\ldots & s & \ldots
\\
\ldots & \perm(s) & \ldots
\\
\ldots & \pile^*(s) & \ldots
\end{array}\right\}
.
\]
(Note that 
the corresponding inverse permutation $\perm^{-1}$ can be obtained in this notation by exchanging the first two rows.)
This column order supports the creation of a \emph{shuffle tableau}:
In the next tableau,
each label $s$ in a row $\pile^*(s)$ and column $\perm(s)$ indicates that
element $s$ is placed into pile $\pile^*(s)$ as the $\perm(s)$-th placement of the deal.
\[
\left.
\begin{tabular}{r|cccccccc}
\diagbox{$\pile^*(s)$}{$\perm(s)$}
	& 1 	& 2 	& 3 	& 4 	& 5 	& 6 	& 7 	& 8	\\
\hline
1	&	&	& 1	&	&	&	& 2	&	\\
2	&	& 3	&	& 4	&	& 5	&	& 6	\\
3	& 7	&	&	&	& 8	&	&	& 
\end{tabular}
\right.
\]

The piles are collected in row order, top-to-bottom.
Queue-like piles are collected left-to-right, whereas stack-like piles would be collected right-to-left.
Collecting the contents of the tableau in this way demonstrates that $\pile^*$ indeed sorts the input on three queues.

\section{Sorting with a pile shuffle on stacks}
\label{sec:stacksort}

Next we consider sorting with stacks, which requires only a minor modification
of our analysis of shuffle on queues.
We present companion results for pile shuffle on stacks to those just derived for queues.

Suppose that $\varperm \in \Perm{n}$ represents the result
of shuffling a deck $\perm \in \Perm{n}$
on \emph{stacks}
using pile assignments $\pile$.
We will denote the relation
\begin{align*}
\label{eq:define-sshuffle}
\sshuffle \, \pile \, \perm = \varperm.
\end{align*}
A shuffle on stacks maintains the same order of the piles as a shuffle on queues, but
within each pile
labels are in the reverse order as if queues were used.
Therefore, the version of~\eqref{eq:qshuffle-order-identity}
which holds for shuffling on stacks is
\begin{equation}
\label{eq:sshuffle-order-identity}
\sshuffle \, \pile \, \perm = \varperm
\qquad \iff \qquad
\indord_s \, \varperm(s)
=
\indord_s \, \Big( \pile(s), \, {-\perm(s)} \Big).
\end{equation}
Somewhat intuitively then,
our results when sorting with stacks are mirror images of those just developed for sorting with queues.
For example,
now each \emph{ascent} of the input permutation requires a new stack
for sort to occur.
%
The proofs of the next results follow the same logic (mirrored) as their counterparts for queues, and
we omit them as exercise(s) for the reader.

\begin{lemma}[Sort with stacks]
\label{lemma:sort-with-stacks}
A pile assignment function $\pile$ is a sort of permutation $\perm \in \Perm{n}$ on stacks
($\sshuffle \, \pile \, \perm = \perm_I$)
if and only if 
\begin{equation}
\label{eq:ssort-cond}
\pile(s+1)
\geq
\pile(s) + \left[ \perm(s+1) > \perm(s) \right]
\qquad \forall s \in [n-1]
.
\end{equation}
\end{lemma}

\begin{lemma}
\label{lemma:ssort-lower-bound}
If $\pile$ is the pile assignments of a stack shuffle sort of permutation $\perm \in \Perm{n}$,
where $\pile(1) = 1$, then
the number of piles $\pile(n)$ satisfies
\begin{equation}
\label{eq:min-stacks}
\pile(n)
\geq
1 + \ascsFn\left(\perm\right)
=
\descrunFn\left(\perm\right)
;
\end{equation}
$\ascsFn$ and $\descrunFn$ are the ascents and descending runs functions respectively;
analogous to descents and ascending runs previously discussed.
\end{lemma}
\begin{lemma}
\label{lemma:ssort-bound-realz}
Given a permutation $\perm \in [n]!$,
let $\pile^*$ be defined by
\begin{equation}
\label{eq:ssort-bound-realz}
\pile^*(s) = 1 + \sum_{s'=1}^{s-1} \left[ \perm(s'+1) > \perm(s') \right]
\qquad \forall s \in [n]
.
\end{equation}
We call it the cumulative descending runs function of $\perm$.
$\pile^*$ is a minimal sort of $\perm$ on stacks.
\end{lemma}

\begin{corollary}
\label{cor:stack-minpiles}
A permutation $\perm$ can be sorted with $\numpile$ stack-like piles if and only if
$\descrunFn(\perm) \leq \numpile$.
\end{corollary}

\section{Sorting random permutations in piles}
\label{sec:sort-random}

\makecommand{\prob}{{\mathbb P}}
\makecommand{\expect}{{\mathbb E}}
\makecommand{\normdist}{{\mathcal N}}

In this brief section we consider the feasibility of sorting random permutations with pile shuffle.
In particular, we examine what is the probability that $\numpile$ piles can sort a permutation $\perm_n$
sampled uniformly from $\Perm{n}$.

According to Corollary~\ref{cor:min-piles-qsort} of the prequel,
$\numpile$ queues can sort a permutation $\perm$
if and only if $\numpile \geq
1 + \descfn(\perm)$;
by Corollary~\ref{cor:stack-minpiles},
$\numpile$ stacks can sort it if $m \geq
1 + \ascsFn(\perm)$.
The number of permutations of length $n$
with $k$ ascents (or descents)
is known as the Eulerian number
$\left\langle
{n \atop k}
\right\rangle$~\cite[p.~4]{bona_combinatorics_2004}.
Eulerian numbers are a cornerstone
of a deep and fascinating study
of permutation statistics~\cite{petersen2015eulerian,butler_stirling-euler-mahonian_2023,foata_denerts_1990}.
Then,
with $\perm_n$ uniformly distributed
over $\Perm{n}$,
$$
\prob\left[ \ascsFn(\perm_n) = k \right]
= 
\prob\left[ \descfn(\perm_n) = k \right]
= \frac{1}{n!}
\left\langle
{n \atop k}
\right\rangle
.
$$
It follows that
\begin{align*}
\prob\left[
\textrm{$\numpile$ queues (stacks) can sort $\perm_n$%
}
\right]
&=
\prob\left[
\descfn(\perm_n) \leq \numpile - 1
\right]
\\
&=
\frac{1}{n!}
\sum_{k=0}^{\numpile - 1}
\left\langle
{n \atop k}
\right\rangle
.
\end{align*}

For large $n$ this probability can be well approximated, because
it is known that Eulerian statistics follow a central limit theorem%
~\cite{hwang_asymptotic_2019}.
For example,
the statistic
$$\frac{
\descfn(\perm_n) - n/2
}{
\sqrt{n/12}
}$$
converges in distribution to the Normal distribution,
in the sense that
$$
\sup_{x\in\reals} \left|
\prob\left(
\frac{\descfn(\perm_n) - n/2}{\sqrt{n / 12}}
\leq x
\right) - \Phi(x)
\right|
\to 0
;
$$
here
$\Phi$ denotes the standard Normal distribution function.
That means both $\descfn(\perm_n)$ and $\ascsFn(\perm_n)$ concentrate around $n/2$, which unfortunately means that
for any given number $\numpile$ of piles,
the probability of sorting $\perm_n$ vanishes as the permutation length $n$ grows.

\section{Sorting with a pile shuffle on queues and stacks}
\label{sec:hetero-sort}

\makecommand{\pileTypeSingle}{x}
\makecommand{\newPileFn}{{\rm sep}}

\makecommand{\numQueue}{q}
\makecommand{\numStack}{s}

\makecommand{\updateQuota}{f}
\makecommand{\updateQueues}{{\updateQuota_{\rm Q}}}
\makecommand{\updateStacks}{{\updateQuota_{\rm S}}}

\makecommand{\dpstate}{{\mathbf x}}

In the prequel, all piles within a given shuffle had the same type, either all queues or all stacks.
Next we consider heterogeneous shuffle, where
the piles may be of different types within a single deal.
We provide:
(1) a mathematical shuffle relation, now on a heterogeneous mixture of queues and stacks,
(2) necessary and sufficient conditions for heterogeneous shuffle to sort an input permutation,
and
(3) efficient construction of a minimal sort, in linear time whenever feasible.

\subsection{Shuffle input-output relation}

In the heterogeneous case,
a shuffle is defined not only by the assignment $\pile$ of labels into piles (as before), but additionally
by an assignment $\pileType$ of types to those piles.
Using an alphabet $\arbpile = \{ \text{$\qsym$[ueue]}, \, \text{$\ssym$[tack]} \}$,
if $\pileType$ denotes the assignment of types to piles then
\[
\pileType(\pile) = \begin{cases}
\qsym, & \textrm{pile $\pile$ is a queue} \\
\ssym, & \textrm{pile $\pile$ is a stack}
.
\end{cases}
\]
Suppose that $\varperm \in \Perm{n}$ represents the result of
shuffling a deck $\perm \in \Perm{n}$ using
pile assignments $\pile$ and
type assignments
$\pileType$.
We denote the relation
\[
\shuffleOp \, \pileType \, \assignFn \, \perm = \varperm
.
\]

It is worth mentioning that the three modes of \emph{shelf shuffle} studied in~\cite{FULMAN2021112448}
map readily into this framework:
One can verify that
the \emph{strict} mode of shelf shuffle on $\numpile$ shelves is the queues-only case $\pileType = \qsym^\numpile$,
\emph{standard} mode is the alternating types case $\pileType = (\ssym\qsym)^\numpile$, and
\emph{lazy} mode is the case $\pileType = \qsym (\ssym\qsym)^\numpile$.

Once again the piles are put into natural order, however 
within each pile
the order is determined by the pile type.
We let $\pileTypeInd$ denote the indicator function for the stack pile type, \ie,
\begin{equation}
\label{eq:pile-type-indicator}
\pileTypeInd(\pile)
= \left[ \pileType(\pile) = \ssym \right]
= \begin{cases}
0	& \pileType(\pile) = \qsym
\\
1	& \pileType(\pile) = \ssym
.
\end{cases}
\end{equation}
Then combining the models for each pile type---%
\eqref{eq:qshuffle-order-identity} and~\eqref{eq:sshuffle-order-identity}---%
we obtain
\begin{align}
\label{eq:pile-shuffle-model-hetero}
\shuffleOp \, \pileType \, \assignFn \, \perm = \varperm
\quad \iff \quad
\indord_s \, \varperm(s)
\, = \,
\indord_s \, \Big(
	\pile(s), \,
	(-1)^{\pileTypeInd(\assignFn(s))} \perm(s)
\Big)
.
\end{align}

\subsection{Sorting shuffles and feasibility}
The condition for a heterogeneous shuffle to sort its input permutation follows the familiar pattern
of~Lemma~\ref{lemma:sort-with-queues}:
\begin{lemma}[Heterogeneous sort]
\label{lemma:hetero-sort}
Let $(\prec_\qsym, \prec_\ssym) = (<, >)$,
\ie,
$i \prec_\qsym j \iff i < j$, and
$i \prec_\ssym j \iff i > j$.
A heterogeneous shuffle $(\pileType, \pile)$ sorts permutation $\perm\in\Perm{n}$
($\shuffleOp \, \pileType \, \assignFn \, \perm = \perm_I$)
if and only if
\begin{equation}
\label{eq:hetero-sort}
\pile(s+1)
\geq
\pile(s) + \left[
	\perm(s+1) \prec_{ \pileType(\pile(s)) } \perm(s)
\right]
\qquad \forall s \in [n-1]
.
\end{equation}
\end{lemma}
\begin{proof}[Proof Sketch]
The result is still proven by the same strategy as that of Lemma~\ref{lemma:sort-with-queues},
but starting from~\eqref{eq:pile-shuffle-model-hetero} rather than~\eqref{eq:sort-order-identity-var}.
Each of the previous cases must now be considered separately
when $\pileType(\pile(s)) = \qsym$ and when $\pileType(\pile(s)) = \ssym$.
However, the additional complexity is minor, and we omit a full proof.
\end{proof}

Note~\eqref{eq:hetero-sort} generalizes the forms of~\eqref{eq:qsort-cond} and~\eqref{eq:ssort-cond}.
Therefore, using essentially the same technique as before,
we squeeze the inequality recurrence~\eqref{eq:hetero-sort} to obtain bounds.
\begin{lemma}
\label{lemma:hetero-func-bound}
Given a pile shuffle $(\pileType, \pile)$
of a permutation $\perm \in [n]!$,
starting w.l.g. from $\pile(1) = 1$,
let 
$\pile^*$ be defined by
\begin{equation}
\label{eq:min-piles-mixed}
\begin{cases}
\pile^*(1) = 1
\\
\pile^*(s+1) = \pile^*(s) + \left[
	\perm(s+1) \prec_{ \pileType(\pile^*(s)) } \perm(s)
\right]	& \forall s \in [n-1]
.
\end{cases}
\end{equation}
If 
$(\pileType, \pile)$
sorts 
$\perm$,
%
then $\pile(s) \geq \pile^*(s)$ for all $s \in [n]$.
\end{lemma}
\begin{proof}
The proof is by induction.
The base case $\pile(1) \geq \pile^*(1) = 1$ is by definition.
We may assume as the inductive hypothesis that
for a given $s \in [n-1]$,
$\pile(s) \geq \pile^*(s)$.
If 
$\pile(s) = \pile^*(s)$, then we obtain $\pile(s+1) \geq \pile^*(s+1)$ immediately, since
the right-hand side of~\eqref{eq:hetero-sort} becomes $\pile^*(s+1)$ by definition.
Otherwise,
$\pile(s) \geq \pile^*(s) + 1$, which we bookend with
two uniform bounds,
$\pile(s+1) \geq \pile(s)$ and
$\pile^*(s) + 1 \geq \pile^*(s+1)$,
to obtain the result.
\end{proof}

\begin{corollary}
\label{cor:min-files-mixed}
\eqref{eq:min-piles-mixed} obtains a minimal sort of $\perm$ on $\pileType$, if a sort on $\pileType$ exists.
\end{corollary}
\begin{proof}
If \eqref{eq:min-piles-mixed} has a solution,
then
$\pile^*$ is a sort because it satisfies~\eqref{eq:hetero-sort} everywhere, and
the number $\pile^*(n)$ of piles used is the minimum
by Lemma~\ref{lemma:hetero-func-bound}.
If any sort $(\pileType, \pile)$ of $\perm$ on $\pileType$ exists, then
$\pile^*$ is well defined.
\end{proof}

Lemma~\ref{lemma:hetero-func-bound} and its corollary are fundamental results of this paper, and
will be relied upon throughout the investigations of the sequel.
They reveal the pivotal role that the type assignments $\pileType$ plays---%
it fully specifies a minimal shuffle $\pile^*$---%
in deciding whether sorting a permutation on a sequence of piles is feasible.
If~\eqref{eq:min-piles-mixed} has no solution---\ie, if $\pileType$ is a finite sequence and too short---then
there can be no solution to~\eqref{eq:hetero-sort} either.
Conversely, if a solution exists, then \eqref{eq:min-piles-mixed} obtains one efficiently in a linear scan.
%
\begin{defn}
\label{defn:sort-with-types}
We say a type schedule $\pileType$ (an assignment of pile types) sorts a permutation $\perm$ if
there exists $\pile$ such that $(\pileType, \pile)$ sorts $\perm$.
\end{defn}
%

\makecommand{\sortScenario}{\pileAssignments}

\subsection{Dealer's choice pile shuffle sort}
\label{subsec:dealer-choice-single}

If the dealer is free
to choose the type of each of $\numpile$ piles arbitrarily during the deal,
then
a given deck may be sorted if and only if there is a type assignment $\pileType \in \arbpile^\numpile$
which sorts it.
%
In principle one could decide feasibility by checking each potential assignment until
they find one for which~\eqref{eq:min-piles-mixed} has a solution.
However,
while each check
can be done in linear time,
the search space $\arbpile^\numpile$ is exponential in the number of piles ($2^\numpile$ possible assignments).
%
That could be problematic in practice since, for example,
we have found that 
sorting a random permutation of length $n$ typically requires a non-trivial fraction of $n$ piles.

Fortunately,
a minimal sort can be obtained, if one exists,
in time that is linear in the permutation length,
by
combining~\eqref{eq:min-piles-mixed} with
a greedy solution for choosing pile types
$\pileType^*$:
If label $s$ begins a new pile, then
we should choose the new pile's type so the next item ($s+1$) may be placed there also.
Mathematically,
\begin{equation}
\label{eq:pile-type-look-ahead}
\text{$s = 1$ or $\pile^*(s) > \pile^*(s-1)$}
\implies
\pileType^*(\pile^*(s)) = \begin{cases}
\qsym & \perm(s+1) > \perm(s) \\
\ssym & \perm(s+1) < \perm(s)
.
\end{cases}
\end{equation}
Doing otherwise cannot be advantageous, as we demonstrate next.

\begin{lemma}
The pile shuffle $(\pileType^*, \pile^*)$
defined by~\eqref{eq:min-piles-mixed} and~\eqref{eq:pile-type-look-ahead}
is a minimal sort of input permutation $\perm$:
For any sort $(\pileType, \pile)$ with $\pile(1) = 1$,
$\pile(n) \geq \pile^*(n)$.
\end{lemma}
\begin{proof}
Notably,
the recurrence equations can be applied in the forward direction
to refine any initial sort, \ie, solution $(\pileType, \pile)$ to~\eqref{eq:hetero-sort}.
For each position $s \in [n]$ in order, we do the following:
First, change $\pile(s)$ to $\pile^*(s)$, potentially reducing it ($\pile(s) \geq \pile^*(s)$ by Lemma~\ref{lemma:hetero-func-bound}).
Next, 
modify $\pileType$ at $\pile^*(s)$ if needed so that~\eqref{eq:pile-type-look-ahead} holds there.
Neither operation can cause~\eqref{eq:hetero-sort} to become violated, so
the pair $(\pileType, \pile)$ remains a valid sort after each iteration of this process.
Meanwhile, the number of piles $\pile(n)$ can never increase.

Since the pair of equations fully specifies $(\pileType^*, \pile^*)$, then
by the end of the procedure
we obtain that solution regardless of starting sort.
In particular, since we obtain the result starting from any minimal sort---and 
without introducing additional piles---then
$(\pileType^*, \pile^*)$ is itself minimal.
\end{proof}

Other formulations of the Dealer's choice pile shuffle are potentially of interest.
In particular, it could be natural to incorporate
separate capacities, \ie inequality constraints, on
the numbers of queues, stacks, and total piles available for shuffle.
Other formulations do not yield as readily to the greedy strategy for arbitrary choice,
and this paper leaves such investigations for future work.

\subsection{Reducing arbitrary shuffle to sort}

Our discussion so far has pertained exclusively to the use of pile shuffle to sort a deck.
However, any desired shuffle can be framed as a sort,
through a transformation that is the subject of this brief section.

\begin{prop}
\label{lemma:general-shuffle-queues}
If $\shuffleOp \, \pileType \, \pile \, \perm = \varperm$, then
$\shuffleOp
	\, \pileType
	\, \left( \pile(\perm') \right)
	\, \left( \perm(\perm') \right)
= \varperm(\perm')$
for any third permutation $\perm' \in \Perm{n}$.
\end{prop}
\begin{proof}
Given a permutation $\perm' \in [n]!$,
we can imagine giving each element of the deck two labels:
a primary label $s \in [n]$, and a secondary label $s' = (\perm')^{-1}(s)$.
Then both $\perm(s)$ and $\perm(\perm'(s'))$ describe the position of the element in the starting deck.
Given a shuffle $(\pileType, \pile)$ by the primary label,
both $\pile(s)$ and $\pile(\perm'(s'))$ describe the same pile that the element is assigned to,
whose type is $\pileType(\pile)$.
Finally,
both $\varperm(s)$ and $\varperm(\perm'(s'))$ describe the position of the element in the resulting deck.
We have effectively shuffled a deck $\perm(\perm')$
with $(\pileType, \pile(\perm'))$,
and
obtained $\varperm(\perm')$.
\end{proof}

\begin{corollary}
\label{cor:general-shuffle-queues-sort}
$\shuffleOp \, \pileType \, \pile \, \perm = \varperm$ if and only if
$\left( \pileType, \pile(\varperm^{-1}) \right)$ is a pile shuffle sort
of $\perm(\varperm^{-1})$.
\end{corollary}
\begin{proof}
This is the special case of Lemma~\ref{lemma:general-shuffle-queues}, with $\perm' = \varperm^{-1}$.
\end{proof}

\section{Sorting in multiple rounds of pile shuffle}
\label{sec:multi-round}

\makecommand{\vrpile}{{\hat\pile}}
\makecommand{\vrPileType}{{\hat\pileType}}
\makecommand{\vrPileTypeFn}{{\hat\pileTypeFn}}

\makecommand{\designpile}{{\tilde\pile}}
\makecommand{\reverseFn}{{\rm rev}}

\makecommand{\pileTypeFn}{\pileTypeInd}
\makecommand{\designPileTypeFn}{{\tilde\pileTypeFn}}

\makecommand{\typeSchedules}{{X}}

In Section~\ref{sec:sort-random}
we discovered that
to sort a random permutation of length $n$,
the number of piles required concentrates around $n/2$
in the all-queues or all-stacks case.
(In the dealer choice case it seems to concentrate to a smaller but still constant fraction of $n$.)
Therefore, for any given capacity $\numpile$ of piles,
our chance of sorting a permutation in a single shuffle vanishes as the length $n$ grows.

In the sequel
we consider the problem of sorting in multiple rounds of shuffle, which
augments the power of a bounded number of piles
at the expense of more time to execute the sort.
In the motivating case of sort with a physical facility (\eg, a table of a specific size)
that is often a necessary trade-off.

\subsection{Modeling}

A multi-round pile shuffle is simply a sequence of basic shuffles where
the output of one round becomes the input of the next one.
A shuffle in $\numphase$ rounds can be modeled by a pair
$(\typeSchedules, \schedules)$
of
a sequence of pile type assignments $\typeSchedules = \left( \pileType_1, \ldots, \pileType_\numphase \right)$ and
a sequence of pile assignments
$\schedules = \left( \pile_1, \ldots, \pile_\numphase \right)$.
These induce a sequence of permutations
(deck states)
according to the recurrence
\begin{align*}
\begin{cases}
\perm_0 = \perm, \\
\perm_{\phase} = \shuffleOp \, \pileType_\phase \, \pile_{\phase} \, \perm_{\phase - 1}
	& \textrm{for $1 \leq \phase \leq \numphase$}, \\
\varperm = \perm_\numphase
;
\end{cases}
\end{align*}
$\varperm \in \Perm{n}$ is the final deck order
resulting by
shuffling a deck starting in order $\perm \in \Perm{n}$ in this way.
We denote the relation
\begin{align*}
\label{eq:define-shuffle-multi}
\shuffleOp \, \typeSchedules \, \schedules \, \perm = \varperm.
\end{align*}

\begin{defn}
\label{defn:hetero-multi-types-sort}
We say a multi-round assignment $\typeSchedules$ of pile types sorts a permutation $\perm$
if
there exists $\schedules$ such that
$\left( \typeSchedules, \schedules \right)$ sorts $\perm$.
\end{defn}

\subsection{Sorting with multiple pile shuffles on queues}
\label{sec:multi-round-queue-sort}

\makecommand{\vrNumPile}{{\hat\numpile}}

\makecommand{\itOne}{i}
\makecommand{\itTwo}{j}

We start by analyzing multi-round shuffle with only queues, \ie,
$$\typeSchedules = \left( \qsym^{\numpile_1}, \ldots, \qsym^{\numpile_\numphase} \right),$$
where
$\numpile_\phase$ denotes the number of queues available in each round $\phase \in [\numphase]$.
We do not assume the number of queues is the same in each round.

We observe that 
if labels $\itOne$ and $\itTwo$ are placed in the last round $\numphase$ 
so that $\pile_\numphase(\itOne) < \pile_\numphase(\itTwo)$, then
$\varperm(\itOne) < \varperm(\itTwo)$.
If they are placed in the same final pile $\pile_\numphase(\itOne) = \pile_\numphase(\itTwo)$, but
in the previous round
they are placed in different piles so that
$\pile_{\numphase-1}(\itOne) < \pile_{\numphase-1}(\itTwo)$, then
$\varperm(\itOne) < \varperm(\itTwo)$.
Continuing in this fashion it is easy to reason that
\begin{equation}
\label{eq:mqshuffle-order-identity}
\indord_s \, \varperm(s)
=
\indord_s
\Big(
\pile_\numphase(s), \pile_{\numphase-1}(s), \ldots, \pile_1(s),
\, \perm(s)
\Big)
.
\end{equation}
The final order component $\perm$ captures the fact that
if two elements $s$ and $t$ are placed in all the same piles,
they retain their original order after shuffle.

If we enumerate
pile assignments from $\pile_\phase(1) = 0$,
instead of $\pile_\phase(1) = 1$,
then 
the sequence $( \pile_\numphase, \ldots, \pile_1 )$ can be viewed as a mixed-radix digital representation
of a number $\vrpile_1$ given by a bijective embedding with recurrence
\begin{equation}
\label{eq:virtual-piles-embedding}
\begin{cases}
\vrpile_\numphase = \pile_\numphase,
\\
\vrpile_\phase = \pile_\phase + \numpile_\phase \vrpile_{\phase+1}
& 1 \leq \phase \leq \numphase - 1
.
\end{cases}
\end{equation}
The embedding has the familiar property of digital representation that
\begin{equation}
\label{eq:hetero-queue-sort-order}
\indord_s \,
\Big(
\pile_\numphase(s), \pile_{\numphase-1}(s), \ldots, \pile_1(s), 
\, \perm(s)
\Big)
=
\indord_s \, \Big( \vrpile_1(s), \, \perm(s) \Big)
.
\end{equation}
The right-hand side of~\eqref{eq:hetero-queue-sort-order} is recognizably the order expression~\eqref{eq:qshuffle-order-identity}
of a single queue shuffle
with $\vrpile_1$ as pile assignments.
In this way the embedding $\vrpile_1$ can be thought of
as assigning the cards of 
the deck 
among a set of ``virtual piles''
in a corresponding single-round queue shuffle scenario.
\begin{lemma}
\label{lemma:multi-phase-queue-sort}
A multi-round shuffle of a permutation $\perm$ on queues, with
pile assignments $\schedules = \left( \pile_\phase \right)_{\phase=1}^\numphase$,
has the same result as a single-round shuffle of $\perm$ on queues with
$\vrpile_1$~\eqref{eq:virtual-piles-embedding}.
\end{lemma}
\begin{proof}
Combining~\eqref{eq:mqshuffle-order-identity} with~\eqref{eq:hetero-queue-sort-order} obtains
$\indord_s \, \varperm(s) = \indord_s \, \Big( \vrpile_1(s), \, \perm(s) \Big)$.
Subsequently, \eqref{eq:qshuffle-order-identity} demonstrates the equivalence to single-round shuffle using $\vrpile_1$.
\end{proof}

Letting
$$\numpile_\typeSchedules := \prod_{\phase=1}^\numphase \numpile_{\phase},$$
\eqref{eq:virtual-piles-embedding} defines $\vrpile_1$
as into a co-domain
$[\numpile_\typeSchedules] - 1$
of $\numpile_\typeSchedules$
virtual piles.
Note that, because
\eqref{eq:virtual-piles-embedding} is reversible,
any shuffle $\vrpile_1$ on up to $\numpile_\typeSchedules$ virtual queues---%
\eg, a minimal sort obtained via~\eqref{eq:qsort-bound-realz}---%
can be transformed readily
into a corresponding multi-round shuffle $\schedules$ on $\typeSchedules$ with the same input-output relation.

\begin{lemma}
\label{lemma:multi-round-piles}
A multi-round queue shuffle
$\typeSchedules = \left( \qsym^{\numpile_1}, \ldots, \qsym^{\numpile_\numphase} \right)$
can sort a permutation $\perm$
(in the sense of Definition~\ref{defn:hetero-multi-types-sort})
if and only if
$\ascruns(\perm) \leq
\prod_{\phase} \numpile_{\phase}
$.
\end{lemma}
\begin{proof}
By Lemma~\ref{lemma:multi-phase-queue-sort},
every shuffle $\schedules$ on $\typeSchedules$ corresponds to a shuffle
$\vrpile_1$~\eqref{eq:virtual-piles-embedding}
on $\leq
\numpile_\typeSchedules
$ queues;
and vice versa.
The result then follows from Corollary~\ref{cor:min-piles-qsort}.
\end{proof}

\begin{corollary}
A permutation $\perm$ can be sorted by $\numpile$ queues in $\numphase \geq 0$ rounds
if and only if
$\ascruns(\perm) \leq \numpile^\numphase$.
\end{corollary}
\begin{proof}
The result is a special case of Lemma~\ref{lemma:multi-round-piles}.
\end{proof}

\subsection{Sorting with multiple pile shuffles on queues and stacks}

\makecommand{\notimplies}{\centernot\implies}

\makecommand{\itOne}{i}
\makecommand{\itTwo}{j}

When we shuffle with a mixture of queues and stacks the order logic is more convoluted.
Now if $\itOne$ and $\itTwo$ are in the same piles in the last $k$ rounds,
the final order of $\itOne$ and $\itTwo$ depends
not only on $\perm_{\numphase-k}$, but also
on the types of those subsequent pile assignments:
%
As we have observed in the prequel,
each round
in which two labels are assigned to the same stack
reverses their order,
whereas assignment to the same queue preserves their order.
Fortunately
these order dynamics remain compatible with a virtual piles interpretation albeit
via a somewhat more convoluted embedding,
\begin{align}
\indord_s \, \varperm(s)
&=
\indord_s \, \Big(
	\designpile_\numphase(s), \ldots, \designpile_1(s), \, \designpile_0(s)
\Big)
\label{eq:hetero-multi-order}
\\
&=
\indord_s \, \left(
	\vrpile_1(s),
	\,
	(-1)^{
		\vrPileTypeFn_1(\vrpile_1(s))
	} \, \perm(s)
\right)
;
\label{eq:multi-round-ord-virtual}
\end{align}
with
order components $\left( \designpile_\phase \right)_{\phase=0}^\numphase$
to be derived presently,
along with $\vrPileTypeFn_1$ and a generalized formula for $\vrpile_1$.

Letting $\pileTypeInd_\phase$ denote the stack indicator function~\eqref{eq:pile-type-indicator}
for the pile types $\pileType_\phase$ in each round $\phase \in [\numphase]$, then
the number
of stacks that a label $s$ is assigned to in
rounds 
$\phase$ through $\numphase$
is given by
$
\sum_{\phase \leq \phase' \leq \numphase} \pileTypeInd_{\phase'}(\pile_{\phase'}(s))
$.
We can denote the parity of that count (zero or one) by
\begin{equation}
\label{eq:suffix-parity}
\designPileTypeFn_\phase(s)
\doteq
\left[
	\sum_{\phase \leq \phase' \leq \numphase} \pileTypeInd_{\phase'}(\pile_{\phase'}(s))
\right] \mod 2
.
\end{equation}
Then based on our discussion,
the order component $\designpile_\phase$
is governed by both $\pile_t$ and $\designPileTypeFn_{\phase+1}$:
Suppose $\pile_\phase(i) < \pile_\phase(j)$ in round $\phase \leq \numphase$,
and $\itOne$ and $\itTwo$ are co-assigned in all future rounds.
%
(If they are in the same piles in all the rounds,
we can
interpret the input permutation $\perm$ as pile assignments $\pile_0$ in an imaginary $t=0$ round.)
If $\itOne$ and $\itTwo$
face an even number of stack assignments in the suffix
($\designPileTypeFn_{\phase+1}(i) = \designPileTypeFn_{\phase+1}(j) = 0$)
then $\varperm(i) < \varperm(j)$;
otherwise, they face an odd number of stack assignments
($\designPileTypeFn_{\phase+1}(i) = \designPileTypeFn_{\phase+1}(j) = 1$)
and
$\varperm(i) > \varperm(j)$.
(Note that 
for the $\phase = \numphase$ case,
$\designPileTypeFn_{\numphase+1} = 0$.)
This behavior is properly captured provided
\begin{equation}
\label{eq:hetero-order-req}
\indord_s \, \designpile_\phase(s)
=
\indord_s \, (-1)^{\designPileTypeFn_{\phase+1}(s)} \, \pile_\phase(s)
,
\qquad 0 \leq \phase \leq \numphase
.
\end{equation}
\eqref{eq:hetero-order-req}
ensures, in the previous scenario, that
$\itOne$ and $\itTwo$ are indistinguishable
according to the order prefix
$(\designpile_\numphase, \ldots, \designpile_{\phase+1})$,
and that
$\varperm(i) < \varperm(j) \iff \designpile_\phase(i) < \designpile_\phase(j)$.
Something to bear in mind is that two labels with the same order component in a given round may be located
in two different \emph{actual} piles of potentially differing type, i.e.
$
\designpile_\phase(\itOne) = \designpile_\phase(\itTwo)
\notimplies
\pile_\phase(\itOne) = \pile_\phase(\itTwo)
$,
or even
$\pileType_\phase(\pile_\phase(\itOne)) = \pileType_\phase(\pile_\phase(\itTwo))$.

Any choice of $\left( \designpile_\phase \right)_{\phase=0}^\numphase$ satisfying~\eqref{eq:hetero-order-req}
would characterize sequential shuffle through \eqref{eq:hetero-multi-order}.
However,
to obtain a virtual piles embedding of those order components
as in the queues-only case,
\ie,
$\vrpile_1$ satisfying~\eqref{eq:multi-round-ord-virtual} and given by
\begin{align}
\label{eq:hetero-multi-embedding}
\begin{cases}
\vrpile_\numphase = \designpile_\numphase,
\\
\vrpile_\phase =
	\designpile_\phase
	+
	\numpile_\phase \vrpile_{\phase+1}
	,
& 1 \leq \phase \leq \numphase-1
,
\end{cases}
\end{align}
each $\designpile_\phase$ must take value on the range $[\numpile_\phase]-1$.
To accomplish that, we introduce the interval-reversing function
$\reverseFn_m(y) = -y + m - 1$;
it maps the interval $[\numpile]-1$ to itself in reverse order.
Letting $\reverseFn_m^n$ denote the composition of $\reverseFn_m$ $n$ times,
it is easy to check $\reverseFn_m^{2k}$ is identity for all $k \geq 0$, and
$\reverseFn_m^{2k+1} = \reverseFn_m$;
and therefore that for any $\pile$ and $n\geq 0$,
\[
\indord_s \, \reverseFn_m^n(\pile(s))
=
\indord_s \, (-1)^{n} \, \pile(s)
.
\]
Therefore,
defining
\begin{equation}
\label{eq:hetero-order-component}
\designpile_\phase(s)
=
\reverseFn_{\numpile_\phase}^{\designPileTypeFn_{\phase+1}(s)}(\pile_\phase(s))
\end{equation}
satisfies our order requirements~\eqref{eq:hetero-order-req}.
The first order component is simply the last round pile assignment,
$\designpile_\numphase = \pile_\numphase$,
since there are no future rounds to cause reversals
(again, $\designPileTypeFn_{\numphase+1} = 0$).
Now
\eqref{eq:hetero-multi-embedding} is bijective,
and defines $\vrpile_\phase$
for each $\phase\in [\numphase]$
as into
a co-domain
$[\vrNumPile_{\phase}] - 1$, 
where
$$ 
\vrNumPile_\phase \doteq \prod_{\phase'=\phase}^\numphase \numpile_{\phase'}
.
$$
(Note therefore, $\vrNumPile_1 = \numpile_\typeSchedules$.)

For the final component of our virtual piles interpretation of mixed multi-round shuffle,
we introduce the new set of indicator functions $\{ \vrPileTypeFn_\phase \}$, defined by
\begin{align}
\label{eq:phase-indicators}
\vrPileTypeFn_\phase(\vrpile_\phase)
= \designPileTypeFn_\phase,
\qquad 
1 \leq \phase \leq \numphase
.
\end{align}
We are motivated to do so
by the observation---%
comparing~\eqref{eq:multi-round-ord-virtual} to~\eqref{eq:pile-shuffle-model-hetero}---%
that $\vrPileTypeFn_1$ can be interpreted as the indicator function
(again in the sense of~\eqref{eq:pile-type-indicator})
associated with some assignment $\vrPileType_1$ of types to the virtual piles;
then:

\begin{prop}
\label{prop:hetero-multi-sort}
A heterogeneous multi-round shuffle $(\typeSchedules, \schedules)$
has the same effect as 
the single-round heterogeneous virtual shuffle $(\vrPileType_1, \vrpile_1)$.
\end{prop}
Remarkably,
we can fully precompute
$\vrPileTypeFn_1$---and therefore virtual types $\vrPileType_1$---%
given only the sequence $\typeSchedules$ of type assignments.
(We derive the equations
at the end of this section.)
Therefore, Prop.~\ref{prop:hetero-multi-sort} brings to bear
the full power of Lemma~\ref{lemma:hetero-func-bound},
generalizing its role in deciding sort feasibility to the multi-round case
through our virtual piles interpretation.
Moreover,
a minimal sort on $\typeSchedules$
can be computed again in linear time if feasible:
First, we produce $\vrPileType_1$ (up to the length $n$ of the permutation)
and obtain $\vrpile^*_1$ if feasible using the single-round solution~\eqref{eq:min-piles-mixed}.
We can then transform $\vrpile^*_1$ into a multi-round assignment
$\schedules^*$ by reversing~\eqref{eq:hetero-multi-embedding}
followed by~\eqref{eq:hetero-order-component}.
The whole process can be done in $O(n)$ time.

\emph{Precomputing $\vrPileTypeFn_1$}:
To end the section we derive the equations to compute $\vrPileTypeFn_1$
(and thereby $\vrPileType_1$)
given the sequence of pile type assignments $\typeSchedules$:

To begin, we can write~\eqref{eq:suffix-parity} as a recurrence
\begin{align}
\label{eq:suffix-parity-rec}
\begin{cases}
\designPileTypeFn_\numphase = \pileTypeFn_\numphase(\pile_\numphase)
\\
\designPileTypeFn_\phase
= \pileTypeFn_\phase(\pile_\phase) \xor \designPileTypeFn_{\phase+1},
& 1 \leq \phase \leq \numphase - 1
.
\end{cases}
\end{align}
Then, by substituting~\eqref{eq:phase-indicators} into~\eqref{eq:suffix-parity-rec}---%
and given $\vrpile_\numphase = \designpile_\numphase = \pile_\numphase$---%
we obtain
\begin{align}
\begin{cases}
\vrPileTypeFn_\numphase = \pileTypeFn_\numphase
\\
\vrPileTypeFn_\phase(\vrpile_\phase)
= \pileTypeFn_\phase(\pile_\phase) \xor \vrPileTypeFn_{\phase+1}( \vrpile_{\phase+1} ),
& 1 \leq \phase \leq \numphase - 1
.
\label{eq:virtual-pile-type-recur-1}
\end{cases}
\end{align}
Finally we substitute expressions for $\vrpile_\phase$ and $\pile_\phase$ in~\eqref{eq:virtual-pile-type-recur-1}:
Namely,
\eqref{eq:hetero-multi-embedding} provides
$
\vrpile_\phase = 
\designpile_\phase
+
\numpile_\phase \vrpile_{\phase+1} 
$,
and
\eqref{eq:hetero-order-component} can be reversed to obtain
$
\makecommand{\arg}{\vrPileTypeFn_{\phase+1}(\vrpile_{\phase+1})}
\pile_\phase = \reverseFn_{\numpile_\phase}^{\arg}(\designpile_\phase)
$,
for each of $1 \leq \phase \leq \numphase - 1$.
Substituting these
we obtain
\begin{varscope}
\makecommand{\index}{\phase}
\makecommand{\realParity}{ \vrPileTypeFn_{\index+1}(\vrpile_{\index+1}) }
\begin{align}
\label{eq:type-indicator-recurrence}
\vrPileTypeFn_{\index}(\designpile_\index + \numpile_\index \vrpile_{\index+1})
&=
\pileTypeFn_\index\left(
	\reverseFn_{\numpile_\index}^{\realParity}
	(\designpile_\index)
\right)
\xor
\realParity
, \qquad
1 \leq \phase \leq \numphase - 1
.
\end{align}
\end{varscope}%
Given that $\vrpile_{\phase+1}$ and $\designpile_\phase$ vary
over $[\vrNumPile_{\phase+1}] - 1$ and $[\numpile_\phase] - 1$,
respectively,
\eqref{eq:type-indicator-recurrence} indeed defines
$\vrPileTypeFn_\phase$
for each $\phase \in [\numphase - 1]$---%
purely in terms of $\pileTypeFn_\phase$ and $\vrPileTypeFn_{\phase+1}$---%
over the whole co-domain $[\vrNumPile_\phase] - 1$ of $\vrpile_\phase$.

\makecommand{\invert}{\bar}

\subsection{Sorting with multiple pile shuffles on stacks}

In this brief section 
we analyze multi-round shuffle with only stacks, \ie,
$$\typeSchedules = \left( \ssym^{\numpile_1}, \ldots, \ssym^{\numpile_\numphase} \right),$$
where
$\numpile_\phase$ denotes the number of stacks available in each round $\phase \in [\numphase]$.

In this special case of the general sort, 
for each $\phase \in [\numphase]$,
$\pileTypeFn_\phase = 1$
uniformly over all of $[\numpile_\phase] - 1$.
This reduces the recurrence
of~\eqref{eq:virtual-pile-type-recur-1}
considerably
to
$\vrPileTypeFn_\phase(\vrpile_\phase) = 1 \xor \vrPileTypeFn_{\phase+1}( \vrpile_{\phase+1} )$,
from which
we obtain
$\vrPileTypeFn_\phase = \left( 1 + \numphase - \phase \right) \mod 2$
uniformly over $[\vrNumPile_\phase] - 1$;
in particular,
$\vrPileTypeFn_1 = \numphase \mod 2$.
Therefore,
if the number of rounds $\numphase$ is even, then
$\typeSchedules$ is equivalent to non-sequential pile shuffle
on $\numpile_\typeSchedules$ virtual queues ($\qsym^{\numpile_\typeSchedules}$);
conversely, 
if the number of rounds $\numphase$ is odd, then
it is equivalent to non-sequential pile shuffle on as many virtual stacks ($\ssym^{\numpile_\typeSchedules}$).

\begin{lemma}
\label{lemma:multi-round-stacks}
A multi-round stack shuffle
$\typeSchedules = \left( \ssym^{\numpile_1}, \ldots, \ssym^{\numpile_\numphase} \right)$
can sort a permutation $\perm$
(in the sense of Definition~\ref{defn:hetero-multi-types-sort})
if and only if
either
$\numphase$ is even and
$\ascruns(\perm) \leq \prod_{\phase} \numpile_{\phase}$
or
$\numphase$ is odd and
$\descrunFn(\perm) \leq \prod_{\phase} \numpile_{\phase}$.
\end{lemma}
\begin{proof}
When the number of rounds is even
multi-round stack sort feasibility is governed by the condition of Corollary~\ref{cor:min-piles-qsort}
on the number of the permutation's ascending runs.
When the number of rounds is odd, feasibility is governed by Corollary~\ref{cor:stack-minpiles} instead, and 
the number of descending runs.
\end{proof}

\section{Sorting in multiple
rounds
of Dealer's choice pile shuffle}
\label{sec:dealer-choice-multi}

We end the present study of multi-round pile shuffle
by re-introducing dealer choice 
in the selection of pile types.
In Section~\ref{subsec:dealer-choice-single}
we showed that
sort feasibility
remains tractable
with dealer choice
in the single-round case:
A minimizing assignment of pile types $\pileType^*$ can be obtained
greedily, \ie by~\eqref{eq:pile-type-look-ahead},
from a set
whose size is exponential in the number of piles.
Meanwhile
feasibility
remains tractable in the multi-round setting on any \emph{fixed} sequences of pile types,
namely by the reduction to single-round shuffle of Prop.~\ref{prop:hetero-multi-sort}.

In this section---and in the forthcoming companion paper---%
we are motivated by two problems of general interest:
\begin{probenv}[Repeated-round Dealer's choice pile shuffle sort]
\label{prob:repeated-round}
Decide whether
a permutation $\perm$ can be sorted in $\numphase$ rounds of Dealer's choice shuffle on $\numpile$ piles;
the permutation $\perm$ and parameters $\numphase$ and $\numpile$ are the instance data.
\end{probenv}
Problem~\ref{prob:repeated-round}
captures the motivating scenario of sort in bounded time, in physical piles, on a surface
whose shape does not typically change from round to round.

\begin{probenv}[Variable-round Dealer's choice pile shuffle sort]
\label{prob:variable-round}
Given per-round pile capacities $\left( \numpile_1, \ldots, \numpile_\numphase \right)$,
decide whether a permutation $\perm$ can be sorted in $\numphase$ rounds of Dealer's choice shuffle,
where
$\numpile_\phase$ is the number of piles in each round $\phase \in [\numphase]$.
\end{probenv}
Problem~\ref{prob:variable-round}
generalizes Problem~\ref{prob:repeated-round}---%
the former contains all the instances of the latter---%
in that each round may have different pile capacity.

Unfortunately, dealer choice brings non-trivial complexity to 
the multi-round setting.
For example,
in the case of Problem~\ref{prob:variable-round},
while it is tempting
to try and obtain a solution
from a minimal sort on $\prod_\phase \numpile_\phase$ virtual piles of dealer-chosen types---%
by invoking Prop.~\ref{prop:hetero-multi-sort}---%
that approach is not sound.
That is because
in general
not every one
of the 
virtual pile types
may be chosen independently, since
there are only a generally smaller number
$\sum_{\phase} \numpile_\phase$
of \emph{actual} piles.

In fact, in the aforementioned companion paper 
we will demonstrate
by reductions from the Boolean satisfiability problem (SAT) that
multi-round Dealer's choice pile shuffle sort (feasibility)
can be
\NPHard/ in general.
In particular, the main result of the forthcoming paper is the following proposition:
\begin{prop}
\label{thm:variable-round-hard}
The variable-round Dealer's choice pile shuffle sort problem is \NPHard/.
\end{prop}
Unfortunately our study leaves as an open question whether repeated-round sort feasibility is \NPHard/.
However, the author's conjecture, which we discuss in the second paper, is that it is.
\begin{conj}
The repeated-round Dealer's choice pile shuffle sort problem is \NPHard/.
\end{conj}
Sort feasibility is generally no harder than NP since
the assignment $\typeSchedules$ of all (real) pile types is a checkable certificate of feasibility.

\subsection{Multi-round mixed sort with a fixed \emph{number} of piles}

Although we are conjecturing that 
the repeated-round Prob.~\ref{prob:variable-round}
is \NPHard/,
if we restrict
it
to an a priori fixed number $\numpile$ of piles,
then a brute force search over all possible assignments decides feasibility in (technically) polynomial time:
In that case,
if 
there are greater than
$\numphase' = \left\lceil \log_\numpile(n) \right\rceil$ rounds, then
a permutation of length $n$ is trivially sortable with $\numpile^{\numphase'} \geq n$ virtual piles.
Otherwise,
we need check only as many as
$\left( 2^\numpile \right)^{\numphase'}
\in O\left( n^{
	(\numpile / \log_2 \numpile)
} \right)$
certificates, each in $O(n)$ time.
Note,
while this approach is technically polynomial-time in the permutation length,
it suffers exponential growth in the number of piles $\numpile$.
This is
reminiscent of
the way SAT is technically polynomial time on any fixed set of variables, even though it is \NPHard/ in general.

\section{Conclusion}
\label{sec:conclusion}

In this paper
we have formalized the pile shuffle, and analyzed its capabilities as a sorting device.
We have characterized the sortable permutations of one or more sequential rounds of pile shuffle
in three variations:
(1) on queues,
(2) on stacks, and
(3) on a heterogeneous mixture of both piles types.
As far as the author is aware, ours is the first formal study to consider
arbitrary mixtures of the two pile types.

We formulated the pile shuffle as a parameterized relation between two permutations.
%
%
The formulation characterizes pile shuffle as a set of functions which,
when paired with permutations representing the starting states of a deck of cards,
inject those cards into an ordered set
reflecting the same label order as 
the deck state after shuffle.
Our characterization is more or less the same thing as $P$-partitions~\cite{FULMAN2021112448},
but stated in a way most conducive to our analysis.

The first part of the paper focused on the single-round case.
Analyzing our model,
we developed necessary and sufficient conditions~\eqref{eq:hetero-sort}
for a given shuffle to sort an input permutation.
Notably, these conditions can be computed efficiently
in a single scan of
the permutation.
%
From these we derived formulas,
in each case,
for the minimum number of piles required for sort, 
and to transcribe a sort on the minimum number of piles.
%
In the homogeneous all-queues or all-stacks cases,
the bounds are in terms of well-studied ascent and descent permutation statistics with deep connections to combinatorics.
In the case that
the dealer is allowed to choose the types of the piles arbitrarily during the shuffle,
we augmented the scan for a minimal sort on fixed pile types~\eqref{eq:min-piles-mixed}
with a greedy method for choosing those types~\eqref{eq:pile-type-look-ahead},
and demonstrated that it obtains a minimal sort in the presence of dealer choice.

The second key contribution of the paper is
a mathematical reduction of
multi-round shuffle on fixed pile types,
into
single-round shuffle on fixed-type ``virtual piles''.
The reduction lifts all of our results about fixed-type single-round pile shuffle to the multi-round setting.
It has the fundamental property that
the number of piles
available for the virtual shuffle 
is equal to the product of the number of piles
available in each round of the actual shuffle.
Therefore, we confirm that repetition augments the power of pile shuffle exponentially, in that
$m$ piles over $T$ rounds has the capacity of $m^T$ piles in a single round.
Feasibility of multi-round shuffle sort
on fixed pile types
remains decidable in linear time, because
the formula of our reduction~\eqref{eq:type-indicator-recurrence}
can be truncated to the number of piles used,
which should not be more than the size of the permutation.

In a sudden halt to the momentum of our study,
we discovered that 
dealer choice brings non-trivial complexity to
sort feasibility in the multi-round setting.
We introduced two sort feasibility problems of multi-round Dealer's choice pile shuffle,
and motivated a forthcoming companion paper
in which we demonstrate,
by a novel reduction from Boolean satisfiability (SAT), that
at least one of those variants is NP-Hard.

\bibliographystyle{unsrt}
\bibliography{main}


\end{document}